\documentclass{article}
\usepackage{graphicx} 
\usepackage[utf8]{inputenc}
\usepackage{amsmath, amssymb, amsthm}
\usepackage{geometry}
\usepackage{authblk}
\usepackage{tikz}
\usetikzlibrary{positioning, shapes.geometric, arrows.meta, calc}
\usepackage{hyperref}
\usepackage{enumitem}

\theoremstyle{plain}
\newtheorem{theorem}{Theorem}[section]
\newtheorem{lemma}[theorem]{Lemma}

\newtheorem{proposition}[theorem]{Proposition}
\newtheorem{corollary}[theorem]{Corollary}
\newtheorem{innerclaim}{Claim}

\theoremstyle{definition}

\theoremstyle{remark}

\usepackage{todonotes}
\setuptodonotes{inline}

\title{The Stable Marriage Polytope is Box-Totally Dual Integral}

\author[1]{Michele Barbato%
  \thanks{\href{mailto:michele.barbato@unimi.it}
  {\nolinkurl{michele.barbato@unimi.it}}}}
\author[2]{Emiliano Lancini%
  \thanks{\href{mailto:emiliano.lancini@lamsade.dauphine.fr}
  {\nolinkurl{emiliano.lancini@lamsade.dauphine.fr}}}}
\affil[1]{Department of Computer Science, Università degli Studi di Milano, Milano, Via Celoria 18, 20133, Italy}
\affil[2]{LAMSADE, CNRS UMR 7243,
Université Paris Dauphine--PSL, Paris, France}
\date{ }
\begin{document}
	
	\maketitle
	
	\begin{abstract}
	A stable matching in a bipartite preference system is a matching that has no unmatched pair whose endpoints both prefer each other to their current assignments. Rothblum gave a linear description of the stable marriage polytope, the convex hull of the incidence vectors of stable matchings. We prove that this polytope is box-totally dual integral. The proof uses the affine representation of the stable marriage polytope as an order polytope together with the matricial characterization of box-TDI polyhedra.
 Since Király and Pap proved that Rothblum’s description is totally dual integral, our result further implies that Rothblum’s system is box-totally dual integral.
	\end{abstract}
	
	\section{Introduction}

Given a bipartite graph whose vertices have strict preferences over their edges, the stable marriage problem asks for a matching that is stable: no edge outside the matching should have endpoints that both prefer this edge to their current assignment, with being unmatched regarded as worse than any acceptable partner. This problem was first introduced by Gale and Shapley \cite{GaleShapley1962}. Irving and Leather \cite{IrvingLeather1986} introduced the rotation-poset representation of stable matchings, which subsequently led to efficient algorithms for optimization over stable matchings \cite{IrvingLeatherGusfield1987}. On the polyhedral side, Vande Vate \cite{VandeVate1989} gave a linear description of the stable marriage polytope for complete balanced bipartite graphs. Rothblum \cite{Rothblum1992}, and later Roth, Rothblum, and Vande Vate \cite{RothRothblumVandeVate1993}, extended the polyhedral theory to general bipartite preference systems. Further investigations concerned the facial and linear-inequality structure of the stable marriage polytope \cite{AbeledoRothblum1994,Ratier1996}. More recent work studied its dimension and facets \cite{EirinakisMagosMourtosMiliotis2014} and established its affine relation with the order polytope of the rotation poset \cite{AprileCevallosFaenza2018}. The polyhedral approach was also extended beyond the one-to-one setting: Ba\"iou and Balinski \cite{BaiouBalinski2000} described the stable admissions polytope, while Fleiner  \cite{Fleiner2003} developed \(b\)-matching generalizations.

This work focuses on box-total dual integrality of the stable marriage polytope.
A linear system $Ax\le b$ is \emph{totally dual integral} (TDI) if for any integer cost function for which the optimum is finite, its dual admits an integer optimal solution. 
A system is \emph{box-totally dual integral} (box-TDI) if it remains TDI under the addition of any rational bound on the variables.  A \emph{box-TDI polyhedron} is a polyhedron that can be described by a box-TDI system. Chervet, Grappe, and Robert \cite{ChervetGrappeRobert2021} characterized box-TDI polyhedra in terms of matrices describing the affine hulls of their faces. For a survey of classical results and recent developments, we refer to Lancini and Pisanu \cite{LanciniPisanu2026}. Matching polytopes are among the most extensively studied with respect to box-TDIness, as proven, for example, by the results of Ding, Tan, and Zang~\cite{ding2018} and Grappe, Lacroix, and Pisanu~\cite{grappe2027}. 

\paragraph{Contributions.}
Kir\'aly and Pap \cite{KiralyPap2008} proved that Rothblum's description of the stable marriage polytope is TDI. In this paper, we strengthen this result by proving that the stable marriage polytope is box-TDI for every finite bipartite preference system with strict preferences, and therefore that Rothblum's system is box-TDI. 

Our proof uses the affine representation of the stable marriage polytope as the order polytope of the rotation poset, established by Aprile, Cevallos, and Faenza \cite{AprileCevallosFaenza2018}. We show that the direction space of every face has a totally unimodular basis, and deduce the box-TDIness of this polytope thanks to characterization of Chervet, Grappe, and Robert.

\section{Preliminaries and notation}\label{sec:preliminaries}
\subsection*{Stable matchings}
All graphs are finite, undirected, and simple. Let $G=(U,V;E)$ be a bipartite graph. For a vertex $v$, let $\delta(v)$ denote the set of its incident edges. A \emph{bipartite preference system} $(G,\mathcal O)$ consists of $G$ and a strict total order $<_v$ on $\delta(v)$ for each vertex $v$. We write $e<_v f$ when $v$ prefers $e$ to $f$, and $e\le_v f$ when $e<_v f$ or $e=f$.

An edge $e$ \emph{dominates} an edge $f$ if they share an endpoint $v$ and $e\le_v f$. Let $\varphi(f)$ denote the set of edges that dominate $f$. A matching $M$ is \emph{stable} if every edge is dominated by an edge of $M$. This treats being unmatched as worse than any incident edge. Every bipartite preference system admits a stable matching~\cite{GaleShapley1962}, and all its stable matchings cover the same set of vertices~\cite{RothRothblumVandeVate1993}.

The \emph{stable marriage polytope} $P_{G\mathcal O}$ is the convex hull of the incidence vectors of the stable matchings of $(G,\mathcal O)$. Rothblum~\cite{Rothblum1992} proved that it is described by
\begin{equation}\label{sys:rothblum}
\left\{\begin{array}{ll}
 x(\delta(v))\le 1 & (v\in U\cup V),\\
 x(\varphi(e))\ge 1 & (e\in E),\\
 x_e\ge 0 & (e\in E).
\end{array}\right.
\end{equation}

\subsection*{Box-total dual integrality}
A rational system $Ax\le b$ is \emph{totally dual integral} (TDI) if, for every integer objective vector for which the primal optimum is finite, the dual admits an integer optimal solution. It is \emph{box-totally dual integral} (box-TDI) if $Ax\le b,\ \ell\le x\le u$ is TDI for all rational lower and upper bounds, with absent bounds allowed. A polyhedron is \emph{box-TDI} if it admits a box-TDI description. Cook~\cite{Cook1986} proved that every TDI system describing a box-TDI polyhedron is box-TDI.

A matrix is \emph{totally unimodular} (TU) if every square subdeterminant belongs to $\{0,\pm1\}$. For a nonempty face $F$, let $\operatorname{lin}(F)$ denote the linear space parallel to $\operatorname{aff}(F)$.
Chervet, Grappe, and Robert~\cite{ChervetGrappeRobert2021} characterized box-TDI polyhedra in terms of matricial descriptions.
\begin{theorem}[Chervet, Grappe, and Robert~\cite{ChervetGrappeRobert2021}]\label{thm:boxTDIiffTU}
A rational polyhedron $P$ is box-TDI if and only if, for every nonempty face $F$ of $P$, there exists a totally unimodular matrix whose columns form a basis of $\operatorname{lin}(F)$.
\end{theorem}

Projections onto coordinate subspaces preserve both TDIness and box-TDIness.
\begin{proposition}[Cook~\cite{Cook1986}]\label{prp:projections}
Every projection of a box-TDI polyhedron onto a coordinate subspace is box-TDI.
\end{proposition}

In the proof of Theorem~\ref{thm:PboxTDI} we will use a classical sufficient condition for total unimodularity.
\begin{proposition}[Heller--Tompkins~\cite{HellerTompkins1956}]\label{prp:TUsuff}
A $\{0,\pm1\}$-matrix with at most one $+1$ and at most one $-1$ in each row is totally unimodular.
\end{proposition}

\section{Box-TDIness of the stable marriage polytope}\label{sec:main}
We first prove box-TDIness for complete balanced preference systems using the rotation representation of Aprile, Cevallos, and Faenza~\cite{AprileCevallosFaenza2018}. We then extend the result to general bipartite preference systems.

\subsection{The complete balanced case}\label{sec:complete-case}
Let $(G,\mathcal O)$ be a bipartite preference system with $G=(U,V;E)$ complete and $|U|=|V|$.

We use the standard rotation poset $(\mathcal R,\preceq)$; see~\cite{AprileCevallosFaenza2018,GusfieldIrving1989}. Each rotation $\rho$ has a set $\rho^-$ of discarded edges and a set $\rho^+$ of adopted edges. Let  $B$ the matrix whose columns are the signed incidence vectors of the rotations,    that is $B=[\chi^{\rho^+}-\chi^{\rho^-}:\rho\in\mathcal R\,]$.

We will use the following standard rotation property~\cite{GusfieldIrving1989}:
\begin{equation}\label{eq:rotation-uniqueness}
\text{Each edge is adopted by at most one rotation and discarded by at most one rotation.}
\end{equation}

Write $\rho\lessdot\sigma$ when $\sigma$ covers $\rho$, that is, $\rho\prec\sigma$ and there is no $\tau$ with $\rho\prec\tau\prec\sigma$. In the convention corresponding to ideals, the \emph{order polytope} is~\cite{Stanley1986} $Q_{\mathcal R} =\{z\in[0,1]^{\mathcal R}:z_\rho\ge z_\sigma$  whenever $\rho\lessdot\sigma\}.$
Let $M_U$ denote the $U$-optimal stable matching. Aprile, Cevallos, and Faenza~\cite{AprileCevallosFaenza2018} proved that $B$ has linearly independent columns and that
\[
 \pi:Q_{\mathcal R}\longrightarrow P_{G\mathcal{O}},
 \qquad \pi(z)=\chi^{M_U}+Bz
\]
is an affine isomorphism.

Let $Q'$ be a nonempty face of $Q_{\mathcal R}$. Define an undirected graph $H(Q')$ on $\mathcal R\cup\{\mathbf0,\mathbf1\}$ by joining
\begin{itemize}
\item $\rho$ to $\mathbf0$ if $z_\rho=0$ for every $z\in Q'$;
 \item $\rho$ to $\mathbf1$ if $z_\rho=1$ for every $z\in Q'$;
 \item $\rho$ to $\sigma$ if $\rho\lessdot\sigma$ and $z_\rho=z_\sigma$ for every $z\in Q'$.
\end{itemize}
The two additional vertices lie in distinct components, since $Q'$ is nonempty. Let $C_1,\ldots,C_t$ be the components containing neither of them. We call these the \emph{free rotation blocks} of $Q'$.

\begin{lemma}\label{lem:free-block-basis}
The vectors $\chi^{C_1},\ldots,\chi^{C_t}$ form a basis of $\operatorname{lin}(Q')$.
\end{lemma}
\begin{proof}
A vector $h\in\mathbb R^{\mathcal R}$ belongs to $\operatorname{lin}(Q')$ if and only if it satisfies the homogeneous equalities corresponding to the constraints tight in $Q'$. These require $h$ to vanish on components containing $\mathbf0$ or $\mathbf1$ and to be constant on each free block. Hence
\[
 \operatorname{lin}(Q')
 =\operatorname{span}\{\chi^{C_1},\ldots,\chi^{C_t}\}.
\]
The vectors are linearly independent because their supports are nonempty and pairwise disjoint.
\end{proof}

\begin{theorem}\label{thm:PboxTDI}
The stable marriage polytope of a complete balanced bipartite preference system is box-TDI.
\end{theorem}
\begin{proof}
Let $F$ be a nonempty face of $P_{G\mathcal{O}}$. Since $\pi$ is an affine isomorphism, there exists a unique face $Q'$ of $Q_{\mathcal R}$ such that $F=\pi(Q')$, and we have that $\operatorname{lin}(F)=B\operatorname{lin}(Q')$.
Let $C_1,\dots,C_t$ be the free rotation blocks of $Q'$, and set $Z_F = B[\,\chi^{C_1}\ \cdots\ \chi^{C_t}\,]$.

By Lemma~\ref{lem:free-block-basis}, the columns of $Z_F$ span $\operatorname{lin}(F)$. They are linearly independent because $B$ has linearly independent columns. By~\eqref{eq:rotation-uniqueness}, each row of $B$ has at most one $+1$ and at most one $-1$. Since the blocks $C_1,\dots,C_t$ are disjoint, the same holds for each row of $Z_F$. Thus $Z_F$ is totally unimodular by Proposition~\ref{prp:TUsuff}. The conclusion follows from Theorem~\ref{thm:boxTDIiffTU}.
\end{proof}

\subsection{The general case}
Aprile, Cevallos, and Faenza's results are given for complete balanced bipartite preference systems, but can immediately be extended to the general bipartite case, therefore Theorem \ref{thm:PboxTDI} has full generality.
In this section, we prefer to give an alternative route for the general case, that exploits the polyhedral properties of box-TDIness.

\begin{corollary}\label{cor:general}
The stable marriage polytope of a bipartite preference system is box-TDI.
\end{corollary}
\begin{proof}
Let $(G,\mathcal O)$ be a bipartite preference system with $G=(U,V;E)$. All its stable matchings cover the same set of vertices. For each vertex $v$ unmatched in these matchings, add a vertex $\bar v$ in the opposite part. Complete the enlarged graph to a complete bipartite graph $\bar G$. Preserve the original orders and rank every new edge below every original edge at its endpoints. At each originally unmatched $v$, rank $v\bar v$ above the other new edges; at $\bar v$, rank it first. Order the remaining new edges arbitrarily.
By construction, $(\bar G,\bar{\mathcal O})$ is complete and balanced.

Every original stable matching $M$ extends to a stable matching $\bar M$ by adding all the edges $v\bar v$. Indeed, an original edge outside $M$ is dominated by an edge of $M$, and every new edge outside $\bar M$ is less preferred than the assigned edge at its endpoints. Conversely, restricting any stable matching of $(\bar G,\bar{\mathcal O})$ to $E$ gives an original stable matching.

These two statements imply that $P_{G\mathcal O}$ is the projection of  $P_{\bar G\bar{\mathcal O}}$ onto the edges of $G$.
This latter polytope is box-TDI by Theorem~\ref{thm:PboxTDI}, so Proposition~\ref{prp:projections} proves the claim.
\end{proof}

Combining this result with that of Király and Pap~\cite{KiralyPap2008} we obtain:
\begin{corollary}
Rothblum's description~\eqref{sys:rothblum} of the stable marriage polytope is box-TDI.
\end{corollary}

\section{Concluding remarks}
We proved that the stable marriage polytope is box-TDI.
This does not fully characterize the class of graphs for which the stable matching polytope is box-TDI. For instance, the existence of at most one stable matching for $C_3$ and $C_5$ implies that, for any preference system on these graphs, the corresponding polytope is trivially box-TDI. On the other hand, direct computation on small instances like $K_6$ yields both preference systems whose stable matching polytope is box-TDI and preference systems whose stable matching polytope is not box-TDI. This leaves open the question of characterizing which preference systems give a box-TDI stable matching polytope for a given graph.

\subsection*{Acknowledgment of AI usage}
 The mathematical results and proofs in this paper were obtained independently by the authors, without the use of AI tools. Generative AI tools were subsequently used to assist with the identification of relevant prior literature, editorial feedback on the presentation of the results, and language editing. All mathematical claims, citations, and references were independently checked by the authors, who take full responsibility for the content of the paper.

\bibliographystyle{plain}
\bibliography{references}
\end{document}